\documentclass[OJMO,NoWriteXVIII]{cedram}

\newcommand{\nn}{\mathbb{N}} 
\newcommand{\rr}{\mathbb{R}} 
\newcommand{\real}{\mathbb{R}} 
\newcommand{\norm}[1]{\left\Vert {#1} \right\Vert} 
\newcommand{\erl}{\left(-\infty , +\infty\right]} 
\newcommand{\crit}[1]{\mathrm{crit}\,{#1}} 
\newcommand{\argmin}{\operatorname{argmin}} 
\newcommand{\act}[1]{\left\langle{#1}\right\rangle} 

\newcommand{\bo}{{\bf 0}}
\newcommand{\Seq}[2]{\left\{{#1}^{{#2}}\right\}_{{#2} \in \mathbb{N}}}

\title{Non-Convex Split Feasibility Problems: Models, Algorithms and Theory}

\author{\firstname{Aviv} \lastname{Gibali}}
\address{Department of Mathematics, ORT Braude College, Karmiel 2161002 \\ Israel}
\address{The Center for Mathematics and Scientific Computation, University of Haifa, Mt. Carmel, Haifa 3498838 \\ Israel}
\email{avivg@braude.ac.il}
\author{\firstname{Shoham} \lastname{Sabach}}
\address{Faculty of Industrial Engineering and Management, The Technion, Haifa, 3200003 \\ Israel}
\email{ssabach@ie.technion.ac.il}
\author{\firstname{Sergey} \lastname{Voldman}}
\address{Faculty of Industrial Engineering and Management, The Technion, Haifa, 3200003 \\ Israel}
\email{sergeyv@campus.technion.ac.il}

  
\begin{abstract} 
	In this paper, we propose a catalog of iterative methods for solving the Split Feasibility Problem in the non-convex setting. We study four different optimization formulations of the problem, where each model has advantageous in different settings of the problem. For each model, we study relevant iterative algorithms, some of which are well-known in this area and some are new. All the studied methods, including the well-known CQ Algorithm, are proven to have global convergence guarantees in the non-convex setting under mild conditions on the problem's data.
\end{abstract}

\begin{document}
\maketitle

\section{Introduction}
	In 1994 Censor and Elfving \cite{ce94} introduced the (two-sets) \textit{Split Convex Feasibility} (SCF) Problem, which is formulated as follows. Let $C \subseteq \real^{n}$ and $Q \subseteq \real^{m}$ be two non-empty, closed and convex sets, and let $A : \real^{n} \rightarrow \real^{m}$ be a linear mapping. The SCF Problem then seeks to:
	\begin{equation} \label{P:SCFP}
		\text{find a point } x^{\ast} \in C \text{ such that } y^{\ast} = Ax^{\ast} \in Q.
	\end{equation}
	Here, in this paper, we will focus on the \textit{Split Feasibility} (SF) Problem, where the involved sets are not necessarily convex. For simplicity we start our discussion with only one set per each space and in the last section  discuss how the general multiple-sets case can be treated in a similar way.
\medskip

	The SF Problem and its convex variant were employed in solving several real-world challenging problems, such as in Intensity-Modulated Radiation Therapy (IMRT) treatment planning \cite{cbmt06}, navigation on the Pareto frontier in Multi-Criteria Optimization \cite{gks14} and others.
\medskip

	The main goal of this paper is to study four main formulations of the SF Problem as an optimization problem, where for each model we present relevant existing algorithms and develop new algorithms that can tackle it. For any mentioned algorithm, we discuss the current state-of-the-art convergence guarantee, which in some cases already exists in the literature and in some cases provided for the first time here. We summarize our study in the following table, which will be made fully clear when reading the following section.
\medskip

	\begin{table}[htb]
	\begin{center}
		\begin{tabular}{|c|c|c|c|c|}
			\hline
			Setting & Algorithm & Model & \begin{tabular}[c]{@{}c@{}}Global\\ Convergence\end{tabular} & \begin{tabular}[c]{@{}c@{}}Requirements on the \\  linear transformation $A$\end{tabular} \\ \hline
			\multirow{3}{*}{\begin{tabular}[c]{@{}c@{}}$C$ and \\ $Q$ non-convex\end{tabular}} & Proximal ADMM (Alg. 1) & (SF1) & Unknown & Unknown \\ \cline{2-5}
			& PG  (Alg. 2) & (SF1 - Penalized) & \cite[Th. 5.3, p. 120]{ABS13} & None \\ \cline{2-5}
			& CQ Algorithm (Alg. 4) & (SF1 - Penalized) & Section \ref{SSec:CQ-SF1P} & None \\ \hline
			\multirow{2}{*}{\begin{tabular}[c]{@{}c@{}}$C$ non-convex \\ and $Q$ convex\end{tabular}} & PG (Alg. 5) & (SF3) & \cite[Th. 5.3, p. 120]{ABS13} & None \\ \cline{2-5}
			& AM  (Alg. 3) & (SF1 - Penalized) & Section \ref{SSec:AM-SF1P} & None \\ \hline
			\multirow{2}{*}{\begin{tabular}[c]{@{}c@{}}$C$ convex and \\ $Q$ non-convex\end{tabular}} & Proximal ADMM (Alg. 6) & (SF4) & Section \ref{SSec:Lagrangian} & $AA^T \succ 0$ \\ \cline{2-5}
			& Linear. Prox. ADMM (Alg. 7) & (SF4) & Section \ref{SSec:Lagrangian} & $AA^T \succ 0, \,\,\kappa(A^TA)<2$ \\ \hline
		\end{tabular}
	\end{center}
\end{table}	

	In order to derive the four mentioned formulation, two functions will play a pivotal role: the \textit{indicator function} and the \textit{distance function}. Mathematically speaking, given a non-empty and closed set $D$ in the $d$-dimensional Euclidean space $\real^{d}$, the \textit{indicator function} is defined by
	\begin{equation*}
		\delta_{D}\left(u\right) =
		\begin{cases}
			0, & \text{if } u \in D, \\
			+\infty, & \text{otherwise}.
		\end{cases}
	\end{equation*}
	The \textit{distance function} of the set $D$, denoted by $d_{D}: \real^{d} \rightarrow \real_{+}$, is given by
	\begin{equation*}
		d_{D}\left(u\right) = \min \left\{ \norm{v - u} : \, v \in D \right\}.
	\end{equation*}
	These two functions are obviously lower semi-continuous, non-smooth and, in general, non-convex, unless the set $D$ is convex. Another essential property of these functions, especially when reformulating the SF Problem as an optimization problem, is that both functions are non-negative and achieve their minimal value of zero only on the set $D$. 
\medskip

	Before presenting the models, it is important to note that for the sake of preferable mathematical properties (at least in the convex setting, as we discuss below), the function $d_{D}^{2}\left(\cdot\right)$ is usually considered, rather than the distance function itself. For example, when the set $D$ is convex, then the function $d_{D}^{2}\left(\cdot\right)$ is smooth. This property is very important in deriving solution methods and establishing their theoretical guarantees.
\medskip

	Next we present the four optimization models for the SF Problem, which use the above two functions.
\smallskip

\fbox{\parbox{3in}{\vspace{0.05in}\center{$\text{(SF1)} \quad \min_{x \in \real^{n}} \left\{ \delta_{C}\left(x\right) + \delta_{Q}\left(Ax\right) \right\}$}\vspace{0.05in}}} \qquad \fbox{\parbox{3in}{\vspace{0.05in}\center{$\text{(SF2)} \quad \min_{x \in \real^{n}} \left\{ \frac{1}{2}d_{C}^{2}\left(x\right) + \frac{1}{2}d_{Q}^{2}\left(Ax\right) \right\}$}\vspace{0.05in}}}
\medskip

\fbox{\parbox{3in}{\vspace{0.05in}\center{$\text{(SF3)} \quad \min_{x \in \real^{n}} \left\{ \delta_{C}\left(x\right) + \frac{1}{2}d_{Q}^{2}\left(Ax\right) \right\}$}\vspace{0.05in}}} \qquad \fbox{\parbox{3in}{\vspace{0.05in}\center{$\text{(SF4)} \quad \min_{x \in \real^{n}} \left\{ \frac{1}{2}d_{C}^{2}\left(x\right) + \delta_{Q}\left(Ax\right) \right\}$}\vspace{0.05in}}}
\vspace{0.2in}

	One of the main advantages of formulating the SF Problem as an optimization problem is that there is no need to know/assume in advance consistency of the problem, that is, that the original feasibility problem has a solution. In case that the SF Problem has a solution then the above four models can achieve the minimal value zero. In the convex setting, solutions of the SF Problem coincide with the optimal solution of any of the optimization models. On the other hand, in the non-convex setting, since optimal solutions are very difficult to find, we only know that if the value of the achieved solution is zero then this is a solution of the SF Problem. 
\medskip

	Throughout the paper we will only focus on algorithms which have the following two important properties. First, computational efficiency, by that we mean algorithms which use only orthogonal projections (defined below) onto the individual sets $C$ and $Q$, the mapping $A$ and its transpose. For example, algorithms which use the inverse mapping of $A$ fail in this aspect, due to their computational deficiencies in solving large-scale instances of the SF Problem (see, for instance, \cite{ce94} for such an algorithm). Second, algorithms with a provable convergence theory, and since we focus on the non-convex setting, we mean by that algorithms with convergence guarantees to critical points of the used optimization model (exact details and definitions will be given in the sequel).
\medskip

	In the area of SCF Problems there is one central and well-known algorithm, which was designed to tackle the optimization model (SF3), and is called the CQ Algorithm \cite{Byrne02, Byrne04}. This algorithm, which exploits the structure of the model (SF3) in the convex setting, is nothing else but the classical Projected Gradient (PG) Method. Since, here we focus on the non-convex setting, we point interested readers to \cite{BT2010} and the recent book \cite{B2017-B}, where theoretical results on the PG method can be found, including accelerations and various modifications.
\medskip

	While the literature on iterative methods for solving the SCF Problem is focused on the optimization model (SF3) and to the best of our understanding most algorithms are modifications and/or generalizations of the CQ Algorithm for other settings of the problem or other related problems with similar affinities, there is one recent paper \cite{chen19} (as far as we know) that study the CQ Algorithm in the non-convex setting. We will discuss this work in details, later on, after giving the needed notations and relevant definitions.
\medskip

	Speaking on the different optimization models mentioned above, we have witnessed that most of the works focus on developing algorithms that solve the optimization model (SF3), sometime even without paying attention to this fact. However, another model that got some attention, again only in the convex setting, is (SF2), see for instance \cite{ms07}. In this case, this model poses a completely smooth optimization problem, that can be solved via classical techniques of optimization for smooth and unconstrained problems, such as the Gradient Descent Method, the Newton Method and others (see for example the books \cite{BSS,Bertsekas}). Therefore, in the non-convex setting, this model is less relevant and we actually do not know of any algorithm that can tackle it. To summarize, the models (SF2) and (SF3) are indeed very natural in the convex setting, due to the nice properties of the squared distance function, something that is much less attractive in the non-convex and preclude using relevant	 algorithmic ideas. We will see below that, in the non-convex setting, the model (SF1) plays an important role and even allows to easily recover the CQ Algorithm through different algorithmic idea. Model (SF4) also allows for some algorithms developments as will be given below.
\medskip

	We conclude this part with a short description of the rest of the paper. In Section \ref{Sec:Models} we will discuss each optimization model and present the relevant algorithms (if any) that can tackle it in the non-convex setting. The theoretical guarantees for each algorithm will be presented in Section \ref{Sec:Theory}, where in some cases the convergence result is already known and in other cases the result is new and proven here for the first time. It should be noted that in this paper we focus on theoretical guarantees in the sense of global convergence of the whole generated sequence, in comparison to sub-sequences convergence, to critical points of the tackled model. Another type of theoretical guarantees that can be achieved for such algorithms in the non-convex setting is of a local nature and consist of results of linear rate of convergence. This line of research, which usually requires more information of the involved sets, is not studied here but interested readers can find several recent results in \cite{LukNguTam18} and the references therein. We end the paper, with Section \ref{Sec:Extensions}, where we discuss, in short, the extension of the SF Problem to the multi-set setting.

\section{Algorithms for Solving the SF Problem} \label{Sec:Models}
	The indicator and distance functions, as we discussed above, are our central tools in reformulating the SF Problem as a minimization problem. Another object that is very related to these two functions is the \textit{orthogonal projection} operator. Given a non-empty and closed subset $D$ of the Euclidean space $\real^{d}$, the \textit{orthogonal projection} onto the set $D$ is the operator $P_{D} : \real^{d} \rightarrow D$ which is defined by
	\begin{equation*}
		P_{D}\left(u\right) = \argmin \left\{ \norm{v - u} : \, v \in D \right\}.
	\end{equation*}
	Since the set $D$ is not necessarily convex, the resulting projector $P_{D}$ may be a set-valued operator. In other words, for a given point $u \in \rr^d$, the set $P_D(u)$ consists of all vectors in $D$ that are closest to the point $u$. This connect us to the distance function, defined above, since we obviously have, for all $u \in \real^{d}$, that $d_{D}\left(u\right) = \norm{v - u}$, for any $v \in P_{D}\left(u\right)$.
\medskip

	Recall that in the convex setting, the squared distance function enjoys some nice properties, two of them are given next, and others can be found, for example, in \cite{BC2011-B}.
	\begin{prop}
		Let $D$ be a non-empty, closed and convex subset of $\real^{d}$. Then, the following hold:
		\begin{itemize}
			\item[$\rm{(i)}$] The function $d_{D}^{2}\left(\cdot\right)$ is continuously differentiable.
			\item[$\rm{(ii)}$] The gradient of $\left(1/2\right)d_{D}^{2}\left(\cdot\right)$ is Lipschitz continuous with constant $1$ (nonexpansive) and is given by $I_{d} - P_{D}$, where $I_{d}$ is the identity mapping of $\real^{d}$.
		\end{itemize}    		
	\end{prop}
	In this section we study the four optimization models mentioned above and relevant iterative methods for solving the SF Problem (\ref{P:SCFP}) defined in the Euclidean spaces $\real^{n}$ and $\real^{m}$. Since we are interested in the non-convex setting, we will distinguish in our study between the following three scenarios:
	\begin{itemize}
		\item[$\rm{(i)}$] Both sets $C$ and $Q$ are non-convex.
		\item[$\rm{(ii)}$] The set $C$ is convex and $Q$ is non-convex.
		\item[$\rm{(iii)}$] The set $C$ is non-convex and $Q$ is convex.    		
	\end{itemize}
	Although the last two scenarios include one convex set, the associated optimization models still posses a non-convex minimization. However, due to the partial convexity, we will be able to exploit it in order to design iterative algorithms for solving a relevant model. In order to gain from the convexity of one of the sets, we will consider models that use the squared distance function with respect to this set. See more details below.

\subsection{The Model (SF1)} \label{SSec:ModelSF1}
	We begin with the model (SF1), which tackles both sets $C$ and $Q$ via their associated indicator functions, that is,
	\begin{equation*}
		\min_{x} \left\{ \delta_{C}\left(x\right) + \delta_{Q}\left(Ax\right) \right\}.
	\end{equation*}
	This model fits to the completely non-convex scenario, and all the algorithms to be presented soon do not require any convexity in their development and analysis (exact details will be given in Section \ref{Sec:Theory}). This is a non-smooth and non-convex minimization, which at a first glance seems to be the most complex model between all the four mentioned models, even when both sets are convex. However, the situation is exactly the opposite, and this formulation opens the gate for some interesting algorithms. Moreover, this formulation paves the way to easily obtain the CQ Algorithm in the non-convex setting through a different approach as we will develop below.
\medskip

	The optimization model (SF1) belongs to the well-studied class of additive composite models, which is a very popular and active area of research in optimization, both in the convex and non-convex settings. Putting aside, for the moment, the issue of convexity, the main difficulty comes from the composition of the linear mapping $A$ in the term $\delta_{Q}\left(Ax\right)$. The most popular approach to handle this situation, is methods which decompose the linear mapping $A$ out of the indicator function $\delta_{Q}\left(\cdot\right)$, and therefore usually called decomposition or splitting methods. More precisely, we reformulate Model (SF1) by introducing a new variable $u$, which leads to the following linearly constrained non-smooth minimization
	\begin{equation} \label{SF1-Split_0}
		\min_{x \in \real^{n}, u \in \real^{m}} \left\{ \delta_{C}\left(x\right) + \delta_{Q}\left(u\right) : \, Ax = u \right\}.
	\end{equation}
	This splitting technique usually leads to Lagrangian-based methods. Therefore, the augmented Lagrangian of Problem \eqref{SF1-Split_0} is given by
	\begin{equation*}
		L_{\rho}\left(x , u , y\right) = \delta_{C}\left(x\right) + \delta_{Q}\left(u\right) + \act{y , Ax - u} + \frac{\rho}{2} \norm{Ax - u}^{2},
	\end{equation*}
	where $y \in \real^{m}$ is the multiplier which associated with the linear constraint $Ax = u$, and $\rho > 0$ is a penalization parameter.
\medskip

	The well-known Alternating Directions Method of Multipliers (ADMM) \cite{GM1975,GM1976} is a very popular algorithm for solving large-scale problems with linear constrains (see, for example, a recent review paper on Lagrangian-based methods \cite{ST2019} including the references therein). Here, we will consider the Proximal ADMM (which includes ADMM as a particular case and can be useful in some cases) that applied to the augmented Lagrangian of Problem \eqref{SF1-Split_0} and recorded now.
\medskip

	\noindent\fbox{\parbox{\textwidth}{\textbf{Algorithm 1 - Proximal ADMM for (SF1)} \\ 
		\textbf{Initialization:} Choose $x^{0} \in \real^{n}$ and $u^{0} , y^{0} \in \real^{m}$, and fix $\tau_{1} , \tau_{2} \geq 0$. \\
		\textbf{General Step:} For $k = 0 , 1 , 2 , \ldots$ compute
		\begin{align*}
			u^{k + 1} & \in \argmin_{u} \left\{ \delta_{Q}\left(u\right) + \act{y^{k} , Ax^{k} - u} + \frac{\rho}{2}\norm{Ax^{k} - u}^{2} + \frac{\tau_{1}}{2}\norm{u - u^{k}}^{2} \right\}, \\
			x^{k + 1} & \in \argmin_{x} \left\{ \delta_{C}\left(x\right) + \act{y^{k} , Ax  - u^{k + 1}} + \frac{\rho}{2}\norm{Ax - u^{k + 1}}^{2} + \frac{\tau_{2}}{2}\norm{x - x^{k}}^{2} \right\}, \\
			y^{k + 1} & = y^{k} + \rho\left(Ax^{k + 1} - u^{k + 1}\right).
		\end{align*}}}
\medskip

	However, this algorithm is hard to analyze (to the best of our knowledge there is no convergence theory for this algorithm in the non-convex setting). Therefore, we will not pursue this direction for this model, but discuss Lagrangian-based methods for the model (SF4) below.
\medskip

	Now, we would like to study another way to tackle the optimization model (SF1) and thus we return to Problem \eqref{SF1-Split_0}. We consider now a different approach to tackle the linear constraint. A very popular idea in the literature to handle linear constraints is to penalize them. In this case, the following optimization problem is used
	\begin{equation} \label{SF1-Penalize}
		\min_{x \in \real^{n} , u \in \real^{m}} \left\{ \delta_{C}\left(x\right) + \delta_{Q}\left(u\right) + \frac{\lambda}{2}\norm{Ax - u}^{2} \right\},
	\end{equation}
	where $\lambda > 0$ is a penalizing parameter. This and the previous problems are obviously not equivalent, but it is easy to verify that an obtained solution $x^{\ast}$ is a solution of the SF Problem at hand if and only if the pair $\left(x^{\ast} , Ax^{\ast}\right)$ is an optimal solution of Problem \eqref{SF1-Penalize}, with the optimal value of zero.
\medskip

	The optimization problem \eqref{SF1-Penalize} consists of two blocks, which are coupled through the smooth and convex function $\left(\lambda/2\right)\norm{Ax - u}^{2}$. In addition, the problem includes two separable constraints via the corresponding indicator functions. Such problems belong to the class of block additive composite models, which are well-studied in the convex literature and even in the non-convex there are several works on this class of problems, see, for instance, \cite{BST14}. We present here two approaches -- one that exploits the block structure and the other which does not.
\medskip

	First approach tackles the two blocks as a single block via the Projected Gradient Method, which enjoys the separability in the constraint sets $C$ and $Q$ in order to produce a parallel method. Taking a step-size $\tau > 0$, after simple manipulations we obtain the following algorithm.
\medskip

	\noindent\fbox{\parbox{\textwidth}{\textbf{Algorithm 2 - Projected Gradient for (SF1-Penalized)} \\
		\textbf{Initialization:} Choose $x^{0} \in \real^{n}$ and $u^{0} \in \real^{m}$, and fix $\tau > 0$. \\
		\textbf{General Step:} For $k = 0 , 1 , 2 , \ldots$ compute
		\begin{align*}
			u^{k + 1} & \in P_{Q}\left(u^{k} - \frac{\lambda}{\tau}\left(u^{k} - Ax^{k}\right)\right), \\
			x^{k + 1} & \in P_{C}\left(x^{k} - \frac{\lambda}{\tau}A^{T}\left(Ax^{k} - u^{k}\right)\right).
		\end{align*}}}
\medskip

	In a similar spirit to the application of the PG Method we can actually apply the inertial variant, which yields better performances in practice and tend to converges to critical points with lower function value (more on that can be found in \cite{PS2016}). To keep the paper short, we do not go into details here, but the convergence guarantees we present below also apply to the inertial variant.
\medskip

	Our second approach to tackle the penalized problem \eqref{SF1-Penalize} is designed to exploit the two block structure of the problem. To this end, we will use the Alternating Minimization (AM) technique, where the separability in the constraints is crucial. In this case, we obviously have a cyclic algorithm rather than a parallel as in the PG Method.
\medskip

	\noindent\fbox{\parbox{\textwidth}{\textbf{Algorithm 3 - Alternating Minimization for (SF1-Penalized)} \\
		\textbf{Initialization:} Choose $x^{0} \in \real^{n}$. \\
		\textbf{General Step:} For $k = 0 , 1 , 2 , \ldots$ compute
		\begin{align}
			u^{k + 1} & \in \argmin_{u} \left\{ \delta_{Q}\left(u\right) + \frac{\lambda}{2} \norm{Ax^{k} - u}^{2} \right\}, \label{AMu} \\
			x^{k + 1} & \in \argmin_{x} \left\{ \delta_{C}\left(x\right) + \frac{\lambda}{2}\norm{Ax - u^{k + 1}}^{2} \right\}. \label{AMx}
		\end{align}}}
\medskip

	A careful inspection of the two updating rules reveals a similar computational situation to the Proximal ADMM mentioned above. More precisely, it is easy to see that Problem \eqref{AMu} simply reduces to the computation of the orthogonal projection of $Ax^{k}$ onto the set $Q$. Therefore, as long as the operator $P_{Q}$ is explicitly given or can be efficiently computed, the update of the block $u$ is easy. However, this is not the case with the updating rule of the block $x$. Indeed, Problem \eqref{AMx} is not easily solved unless the matrix $A^{T}A$ is invertible, which is a very demanding assumption that is not assumed in this paper.
\medskip

	We will overcome this obstacle by linearizing the quadratic coupling term. This idea can be seen also as incorporating the AM Algorithm above with a PG step applied on the minimization \eqref{AMx}. More precisely, the suggested algorithm is recorded now, where we denote by $\lambda_{max}\left(A^{T}A\right)$ the largest eigenvalue of the symmetric square matrix $A^{T}A$.
\medskip

	\noindent\fbox{\parbox{\textwidth}{\textbf{Algorithm 4 - Semi Alternating Projected Gradient for (SF1-Penalized)} \\
		\textbf{Initialization:} Choose $x^{0} \in \real^{n}$ and fix $\tau > \lambda_{max}\left(A^{T}A\right)$. \\
		\textbf{General Step:} For $k = 0 , 1 , 2 , \ldots$ compute
		\begin{align}
			u^{k + 1} & \in P_{Q}\left(Ax^{k}\right), \label{SAPGu} \\
			x^{k + 1} & \in \argmin_{x} \left\{ \delta_{C}\left(x\right) + \lambda\act{A^{T}\left(Ax^{k} - u^{k + 1}\right) , x - x^{k}} + \frac{\tau}{2}\norm{x - x^{k}}^{2} \right\}. \label{SAPGx}
		\end{align}}}
\medskip

	Now, the updating rule of the $x$ block is explicitly solvable and by substituting the solution to the $u$ block from \eqref{SAPGu}, we obtain
	\begin{equation*}
		x^{k + 1} \in P_{C}\left(x^{k} - \frac{\lambda}{\tau}A^{T}\left(Ax^{k} - u^{k + 1}\right)\right) = P_{C}\left(x^{k} - \frac{\lambda}{\tau}A^{T}\left(Ax^{k} - P_{Q}\left(Ax^{k}\right)\right)\right).
	\end{equation*}
	The obtained algorithm (after substituting the auxiliary variable $u$) is exactly the CQ Algorithm, which is recovered via application of the Semi Alternating Projected Gradient on the penalized formulation of Problem (SF1). As far as we know this is a new approach to get the CQ Algorithm (rather than applying the PG Method on the Model (SF3), as we discuss in the next subsection), which allow us to derive convergence guarantees under much milder assumptions on the sets $C$ and $Q$ in comparison to \cite{liu18}. 
	
\subsection{The Model (SF3)}
	Now we move to the next optimization model
	\begin{equation*}
		\min_{x} \left\{ \delta_{C}\left(x\right) + \frac{1}{2}d_{Q}^{2}\left(Ax\right) \right\} = \min_{x \in C} d_{Q}^{2}\left(Ax\right),
	\end{equation*}
	where the last equality follows from the property of the indicator function. Meaning, we have a constrained minimization.
\medskip

	As we already discussed before, this model of the SF Problem makes sense mainly if the set $Q$ is convex, since then the function $\left(1/2\right)d_{Q}^{2}\left(\cdot\right)$ is continuously differentiable, and its gradient is given by $I_{m} - P_{Q}$. More importantly, this gradient is Lipschitz continuous with constant $1$. Algorithmically speaking, these properties opens the gate for applying the Projected Gradient (PG) Method, even when the set $C$ is non-convex. This leads again to the well-known CQ Algorithm \cite{Byrne02, Byrne04} via different optimization model and different algorithmic framework (but limited only to the case where $Q$ is convex). 
\medskip

	\noindent\fbox{\parbox{\textwidth}{\textbf{Algorithm 5 - Projected Gradient for (SF3)} \\
		\textbf{Initialization:} Choose $x^{0} \in \real^{n}$ and fix $\tau > 0$. \\
		\textbf{General Step:} For $k = 0 , 1 , 2 , \ldots$ compute
		\begin{align*}
			x^{k + 1} & \in \argmin_{x} \left\{ \delta_{C}\left(x\right) + \act{A^{T}\left(Ax^{k} - P_{Q}\left(Ax^{k}\right)\right) , x - x^{k}} + \frac{\tau}{2}\norm{x - x^{k}}^{2} \right\} \nonumber \\
			& = P_{C}\left(x^{k} - \frac{1}{\tau}A^{T}\left(Ax^{k} - P_{Q}\left(Ax^{k}\right)\right)\right).
		\end{align*}}}
\medskip

	Very recently, the authors of \cite{liu18} managed to derive the CQ Algorithm in the fully non-convex setting via this optimization model using techniques of Difference-of-Convex (DC) Programming. However, in their paper the authors assume two main assumptions on the involved sets $C$ and $Q$, which are not needed in our case. First assumption revolves around the regularity of the distance function $d_{Q}\left(y^{\ast}\right)$, at limit points of the form $y^{\ast} = Ax^{\ast}$, which is needed to obtain the sub-differential of the distance function in the non-convex setting. However, this is a demanding assumption that is difficult to check in practice. In addition, the regularity assumption should be satisfied at limit points of the generated sequence, which complicates more the  ability to verify it (unless all point in $A(C)$ are regular). The second assumption is a prox-regularity of the set $Q$ at limit points of the generated sequence. Verifying this assumption in practice could be also very difficult, and limit significantly the applicability of the results. Under this assumption we assure that the function $d_{Q}^{2}\left(\cdot\right)$ is continuously differentiable, which is crucial property in their global convergence analysis.

\subsection{The Model (SF4)} \label{SSec:ModelSF4}
	In this part we study the Model (SF4), which like the Model (SF3) mixes the indicator and distance functions in the following way
	\begin{equation*}
		\min_{x} \left\{ \frac{1}{2}d_{C}^{2}\left(x\right) + \delta_{Q}\left(Ax\right) \right\},
	\end{equation*}
	and following the same reasons as in the previous part, we assume here that the set $C$ is convex.
\medskip

	Due to the composition of the nonsmooth function $\delta_Q(\cdot)$ with the linear mapping $A$, we can not apply the PG Method on this problem, and would need to consider again splitting techniques. By applying a simple splitting technique, as above, with the introduction of a new variable $u$, we derive an equivalent formulation of the problem
	\begin{equation*}
		\min_{x \in \real^{n} , u \in \real^{m}} \left\{ \frac{1}{2}d_{C}^{2}\left(x\right) + \delta_{Q}\left(u\right) : \, Ax = u \right\}.
	\end{equation*}
	Using the augmented Lagrangian associated with this problem, we can apply the Proximal ADMM to obtain the following algorithm.
\medskip

	\noindent\fbox{\parbox{\textwidth}{\textbf{Algorithm 6 - Proximal ADMM for (SF4)} \\
		\textbf{Initialization:} Choose $x^{0} \in \real^{n}$ and $u^{0} , y^{0} \in \real^{m}$, and fix $\tau \geq 0$. \\
		\textbf{General Step:} For $k = 0 , 1 , 2 , \ldots$ compute
		\begin{align}
			u^{k + 1} & \in \argmin_{u} \left\{ \delta_{Q}\left(u\right) + \act{y^{k} , Ax^{k} - u} + \frac{\rho}{2}\norm{Ax^{k} - u}^{2} \right\}, \label{PADMMu} \\
			x^{k + 1} & = \argmin_{x} \left\{ \frac{1}{2}d_{C}^{2}\left(x\right) + \act{y^{k} , Ax - u^{k + 1}} + \frac{\rho}{2}\norm{Ax - u^{k + 1}}^{2} + \frac{\tau}{2}\norm{x - x^{k}}^{2} \right\}, \label{PADMMx} \\
			y^{k + 1} & = y^{k} + \rho\left(Ax^{k + 1} - u^{k + 1}\right). \label{PADMMy}
		\end{align}}}
\medskip

	First, it is easy to notice that if the parameter $\tau$ is chosen to be zero then we recover the classical ADMM. More importantly, is regarding the updating step of the auxiliary variable $u$. In this case, the augmented term $\left(\rho/2\right)\norm{Ax^{k + 1} - u}^{2}$ serves as a natural proximal term and therefore no need for an additional proximal term. In addition, the sub-problem given in step \eqref{PADMMu} are explicitly solvable and given via the projection onto the set $Q$. Indeed, simple manipulations reveal that
	\begin{equation*}
		u^{k + 1} = P_{Q}\left(\frac{1}{\rho}\left(y^{k} + \rho Ax^{k} + \tau_{1}u^{k}\right)\right).
	\end{equation*}		
	On the other hand, the updating step with respect to the variable $x$ is more involved and requires computing the inverse of the matrix $\rho A^{T}A + \tau I_{n}$ (where $I_{n}$ is the identity matrix), and as mentioned above, could be computationally expansive in large-scale instances.
\medskip

	A classical approach to overcome this obstacle is by linearizing, at iteration $k$, the augmented term $\left(\rho/2\right)\norm{Ax - u^{k}}^{2}$ around the current iteration $x^{k}$ (for more details on this approach in the convex setting we refer to \cite{HY2012,ST2014} and in the non-convex setting to \cite{BST2018}). This leads us to the Linearized Proximal ADMM, which is recorded now.
\medskip

	\noindent\fbox{\parbox{\textwidth}{\textbf{Algorithm 7 - Linearized Proximal ADMM for (SF4)} \\
		\textbf{Initialization:} Choose $x^{0} \in \real^{n}$ and $u^{0} , y^{0} \in \real^{m}$, fix $\tau > 0$. \\
		\textbf{General Step:} For $k = 0 , 1 , 2 , \ldots$ compute
		\begin{align}
			u^{k + 1} & \in P_{Q}\left(\frac{1}{\rho}\left(y^{k} + \rho Ax^{k} + \tau_{1}u^{k}\right)\right). \label{LADMMu} \\
			x^{k + 1} & = \argmin_{x} \left\{ \act{x^{k} - P_{C}\left(x^{k}\right) + \rho A^{T}\left(Ax^{k} - u^{k + 1}\right) , x} +  \act{y^{k} , Ax - u^{k + 1}} + \frac{\tau}{2}\norm{x - x^{k}}^{2} \right\}, \label{LADMMx} \\
			y^{k + 1} & = y^{k} + \rho\left(Ax^{k + 1} - u^{k + 1}\right).\label{LADMMy}
		\end{align}}}
\medskip

	In this case we obviously see that the updating step of the variable $x$ is also explicitly computable and is given by the following formula
	\begin{equation*}
		x^{k + 1} = \left(1 - \frac{1}{\tau}\right)x^{k} + \frac{1}{\tau}P_{C}\left(x^{k}\right) - \frac{1}{\tau}A^{T}\left(y^k + \rho\left(Ax^{k} - u^{k + 1}\right)\right).
	\end{equation*}
	Therefore, in this case we obtain a completely computable algorithm for solving the SF Problem in the partially non-convex setting. As we already mentioned, these two algorithms come with provable convergence guarantees that will be given below in Section \ref{Sec:Theory}.
\medskip

	As the reader probably notice, we did not discuss Model (SF2) so far. The main reason for that is the fact that in the non-convex setting there is no way to exploit properties of this model, which could be handy in designing solution schemes. However, for the SCF Problem, that is, when both sets are convex the situations is better and, as mentioned above, several algorithms can be applied .

\section{Convergence Analysis} \label{Sec:Theory}
	In this section we first recall the table presented in the Introduction and summarize the convergence information regarding all the $7$ algorithms presented in Section \ref{Sec:Models}. The convergence of two algorithm follows directly from known results as stated in table, while one algorithm is lack of convergence theory. Below, we will provide the details about the convergence analysis of all the other four algorithms. Therefore, let us first recall the definition of semi-algebraic sets and functions.

	\begin{defi}[Semi-algebraic sets and functions]
		\begin{itemize}
			\item[$\rm{(i)}$] A subset $S$ of $\real^{d}$ is a real \textit{semi-algebraic set} if there exists a finite number of real polynomial functions $g_{ij} , h_{ij} : \rr^{d} \rightarrow \rr$ such that
				\begin{equation*}
					S = \bigcup_{j = 1}^{p} \bigcap_{i = 1}^{q} \left\{ u \in \rr^{d} : \; g_{ij}\left(u\right) = 0 \, \text{and } \, h_{ij}\left(u\right) < 0 \right\}.
				\end{equation*}
			\item[$\rm{(ii)}$] A function $h : \real^{d} \rightarrow \left(-\infty , +\infty\right]$ is called \textit{semi-algebraic} if its graph
				\begin{equation*}
					\left\{ \left(u , t\right) \in \rr^{d + 1} : \; h\left(u\right) = t \right\}
				\end{equation*}
				is a semi-algebraic subset of $\real^{d + 1}$.
		\end{itemize}
	\end{defi}	
	We are not aiming at reviewing here these notions but few comments could be of help for the reader (more details can be found in \cite{AB2009,ABS13,BDL2006,BDLM2010,BST14} and references therein). First, most of the constraint sets that are used in real-world applications are semi-algebraic, for example, balls with respect to the $\ell_{1}$-norm, $\ell_{2}$-norm, $\ell_{0}$-norm, cone of PSD matrices, Stiefel manifold, to mention just a few. In addition, a very important fact to our study, is that the indicator and distance functions are semi-algebraic whenever the associated set is semi-algebraic.
\medskip

	Semi-algebraicity is mentioned here due to its pivotal role in a recent proof methodology of optimization algorithms in the non-convex setting, which is based on the Kurdyka-{\L}ojasiewicz property \cite{L1963,K1998,BDL2006}. The origin of this proof technique was a decade ago in the paper of Attouch and Bolte \cite{AB2009}. Later on, in \cite{BST14}, the technique was unified and extended into a general proof methodology that can be applied to any algorithm in the non-convex setting as long as the optimization problem at hand involves semi-algebraic objects. Very recently, in \cite{BSTV2018}, a concise version of the methodology was presented and the concept of gradient-like descent sequences was coined. Here, we will mention the slight modification of this concept that was studied first in \cite{STV18}.
	\begin{defi}[Gradient-like Descent Sequence] \label{D:Gradlike}
		Let  $F : \real^{d} \times \real^{p} \rightarrow \erl$ be a proper and lower semicontinuous function which is bounded from below, and let $\left\{ \left(z^{k} , v^{k}\right) \right\}_{k \in \nn}$ be a sequence generated by a certain algorithm for solving the problem
		\begin{equation*}
			\inf \left\{ F\left(z , v\right) : \; z \in \real^{d}, \,\, v \in \real^{p} \right\}.
		\end{equation*}
		We say that $\left\{ \left(z^{k} , v^{k}\right) \right\}_{k \in \nn}$ is \textit{a gradient-like descent sequence} for minimizing $F$ if the following three conditions hold:
		\begin{itemize}
			\item[$\rm{(C1)}$] \textit{Sufficient decrease property.} There exists a positive scalar $\rho_{1}$ such that
				\begin{equation*}
					\rho_{1}\norm{z^{k + 1} - z^{k}}^{2} \leq F\left(z^{k} , v^{k}\right) - F\left(z^{k + 1} , v^{k + 1}\right), \quad \forall \,\, k \in \nn.
				\end{equation*}
			\item[$\rm{(C2)}$] \textit{A subgradient lower bound for the iterates gap.} There exists a positive scalar $\rho_{2}$ such that         	
				\begin{equation*}
					\norm{w^{k + 1}} \leq \rho_{2}\norm{z^{k + 1} - z^{k}}, \quad w^{k + 1} \in \partial F\left(z^{k + 1} , v^{k + 1}\right), \quad \forall \,\, k \in \nn.
				\end{equation*}
			\item[$\rm{(C3)}$] Let $\left(\overline{z} , \overline{v}\right)$ be a limit point of a subsequence $\left\{ \left(z^{k} , v^{k}\right) \right\}_{k \in K}$, then $\limsup_{k \in K \subset \nn} F\left(z^{k} , v^{k}\right) \leq F\left(\overline{z} , \overline{v}\right)$.
		\end{itemize}	
	\end{defi}
	Based on these conditions, global convergence of the generated sequence $\Seq{z}{k}$ to a critical point of $F$ can be proved (see \cite[Theorem 6.2, p. 2149]{BSTV2018}). As usual in the non-convex setting, we will use here the concept of limiting subdifferential for characterizing critical points of $F$ and thanks to Fermat's rule \cite[Theorem 10.1, p. 422]{RW1998}, the set of critical points of $F$ is given by:
	\begin{equation*}
		\crit F = \left\{ \left(z , v\right) \in \real^{d} \times \real^{p} : \;  \bo \in \partial F\left(z , v\right) \right\}.
	\end{equation*}
	Now, we can record the convergence result mentioned above.
	\begin{theo}[Global Convergence] \label{T:AbstrGlob}
		Let $\left\{ \left(z^{k} , v^{k}\right) \right\}_{k \in \nn}$ be a bounded gradient-like descent sequence for minimizing $F$. If $F$ is semi-algebraic (or, in general, satisfies the KL property), then the sequence $\Seq{z}{k}$ converges to some $z^{\ast}$. In addition, for any limit point $v^{\ast}$ of $\Seq{v}{k}$ we have that $\left(z^{\ast} , v^{\ast}\right) \in \crit F$.
	\end{theo}
	Therefore, below we will show that the following algorithms: AM for (SF1 - Penalized) and CQ Algorithm for (SF1 - Penalized), generate gradient-like descent sequences and thus, under the semi-algebraicity assumption, converge to a critical point of the associated optimization model. It is very easy to notice that the proof methodology mentioned above can not be applied to the Lagrangian-based algorithms (they are obviously not descent algorithms) and therefore require a different proof technique that will be discussed below in Section \ref{SSec:Lagrangian}.
\medskip

	One last comment before proving the convergence results of the two algorithms is regarding the boundedness assumption, which is made in Theorem \ref{T:AbstrGlob}. The boundedness assumption on the generated sequence by any of the $7$ algorithms holds in several scenarios such as when the sets $C$ and $Q$ are bounded. For a few more scenarios see \cite{ABRS2010} and also \cite{liu18} which discuss the same setting of SF Problems. Another possibility is to assume that the following condition holds
	\begin{equation*}
		C^{\infty} \cap A^{-1} D^{\infty} = \left\{ \bo \right\},
	\end{equation*}
	which implies the boundedness of the generated sequence (see \cite{liu18}).

\subsection{Convergence of the CQ Algorithm for (SF1 - Penalized)} \label{SSec:CQ-SF1P}
	Recalling the optimization model \eqref{SF1-Penalize}
	\begin{equation*}
		\min_{x \in \real^{n} , u \in \real^{m}} \left\{ F_{1}\left(x , u\right) := \delta_{C}\left(x\right) + \delta_{Q}\left(u\right) + \frac{\lambda}{2}\norm{Ax - u}^{2} \right\}.
	\end{equation*}
	Therefore, from \cite[Exercise 8.8 and Proposition 10.5]{RW1998}, we have that
	\begin{equation*}
		\crit F_{1}\left(x , u\right) = \left(\partial \delta_{C}\left(x\right) + \lambda A^{T}\left(Ax - u\right) , \partial \delta_{Q}\left(u\right) + \lambda\left(u - Ax\right)\right).
	\end{equation*}
	We will show now that the Semi Alternating Projected Gradient method (which is nothing but the CQ Algorithm) generates a globally convergent sequence to critical points of $F_{1}$ under only the semi-algebraicity assumption of the involved sets $C$ and $Q$. 
	\begin{prop}
		Let $\left\{ \left(x^{k} , u^{k}\right) \right\}_{k \in \nn}$ be a bounded sequence generated by the Semi Alternating Projected Gradient method. Suppose that $C$ and $Q$ are semi-algebraic sets. Then, the sequence $\Seq{x}{k}$ converges globally to some point $x^{\ast}$ such that for any limit point $u^{\ast}$ of the sequence $\Seq{u}{k}$, the pair $\left(x^{\ast} , u^{\ast}\right)$ is a critical point of $F_{1}$.
	\end{prop}
	\begin{proof}
		We first prove that the sequence $\left\{ \left(x^{k} , u^{k}\right) \right\}_{k \in \nn}$ is a a gradient-like descent sequence for minimizing $F_{1}$ according to Definition \ref{D:Gradlike}. We begin with the first condition (C1). From the Pythagoras identity we have that
		\begin{align*}
			\frac{1}{2}\norm{Ax^{k + 1} - u^{k + 1}}^{2} & = \frac{1}{2}\norm{Ax^{k} - u^{k + 1}}^{2} + \act{Ax^{k} - u^{k + 1} , Ax^{k + 1} - Ax^{k}} + \frac{1}{2}\norm{Ax^{k + 1} - Ax^{k}}^{2} \\
			& \leq \frac{1}{2}\norm{Ax^{k} - u^{k + 1}}^{2} + \act{Ax^{k} - u^{k + 1} , Ax^{k + 1} - Ax^{k}} + \frac{\lambda_{max}\left(A^{T}A\right)}{2}\norm{x^{k + 1} - x^{k}}^{2}.
		\end{align*}
		Writing the optimality condition of \eqref{SAPGx} yields that
		\begin{equation*}
			\delta_{C}\left(x^{k + 1}\right) + \lambda\act{Ax^{k} - u^{k + 1} , Ax^{k + 1} - Ax^{k}} + \frac{\tau}{2}\norm{x^{k + 1} - x^{k}}^{2} \leq \delta_{C}\left(x^{k}\right).
		\end{equation*}
		By combining the last two inequalities we obtain that
		\begin{equation*}
			\frac{\tau - \lambda \cdot \lambda_{max}\left(A^{T}A\right)}{2}\norm{x^{k + 1} - x^{k}}^{2} + F\left(x^{k + 1} , u^{k + 1}\right) \leq F\left(x^{k} , u^{k + 1}\right).
		\end{equation*}
		From the optimality of $u^{k + 1}$ in step \eqref{SAPGu} we obtain
		\begin{equation*}
			F\left(x^{k} , u^{k + 1}\right) \leq F\left(x^{k} , u^{k}\right),
		\end{equation*}
		and therefore from the last two inequalities the first assertion holds.
	
		For proving condition (C2) we use the optimality conditions associated with the two steps \eqref{SAPGu} and \eqref{SAPGx}, meaning
		\begin{equation*}
			\lambda\left(Ax^{k} - u^{k + 1}\right) \in \partial \delta_{Q}\left(u^{k + 1}\right)	,	
		\end{equation*}
		and
		\begin{equation*}
			\lambda A^{T}\left(u^{k + 1} - Ax^{k}\right) + \tau\left(x^{k} - x^{k + 1}\right) \in \partial \delta_{C}\left(x^{k + 1}\right).
		\end{equation*}
		Therefore, by subtracting $\lambda Ax^{k + 1}$ and adding $\lambda A^{T}Ax^{k + 1}$, respectively, from both sides of each inclusion, we arrive at
		\begin{equation*}
			\lambda\left(Ax^{k} - Ax^{k + 1}\right) \in \partial \delta_{Q}\left(u^{k + 1}\right) + \lambda\left(u^{k + 1} - Ax^{k + 1}\right),	
		\end{equation*}
		and
		\begin{equation*}
			\lambda A^{T}\left(Ax^{k + 1} - Ax^{k}\right) + \tau\left(x^{k} - x^{k + 1}\right) \in \partial \delta_{C}\left(x^{k + 1}\right) + \lambda A^{T}\left(Ax^{k + 1} - u^{k + 1}\right).
		\end{equation*}
		This proves that $w^{k + 1} := \left(\left(\lambda A^{T}A - \tau I\right)\left(x^{k + 1} - x^{k}\right) , \lambda A\left(x^{k} - x^{k + 1}\right)\right) \in \partial F\left(x^{k + 1} , u^{k + 1}\right)$. Hence, the assertion follows.
	
		Lastly, condition (C3) easily follows since on the generated sequence we have that $F_{1}\left(x^{k} , u^{k}\right) = \left(\lambda/2\right)\norm{Ax^{k} - u^{k}}^{2}$. Therefore $F_{1}$ is continuous on the generated sequence, which proves the required assertion.
		
		Now, from Theorem \ref{T:AbstrGlob} we get the desired result.
	\end{proof}

\subsection{Convergence of AM for (SF1 - Penalized)} \label{SSec:AM-SF1P}
	Here we again focus on the model \eqref{SF1-Penalize} and therefore use again the function $F_{1}\left(x , u\right) = \delta_{C}\left(x\right) + \delta_{Q}\left(u\right) + \left(\lambda/2\right)\norm{Ax - u}^{2}$, but with the additional assumption that the set $Q$ is convex.
\medskip

	For simplicity of the proof we will analyze the following shifted equivalent version of the algorithm.
\medskip

	\noindent\fbox{\parbox{\textwidth}{\textbf{Alternating Minimization for (SF1-Penalized)} \\
		\textbf{Initialization:} Choose $x^{0} \in \real^{n}$ and set $u^{1} = P_{Q}\left(Ax^{0}\right)$. \\
		\textbf{General Step:} For $k = 1 , 2 , \ldots$ compute
		\begin{align}
			x^{k + 1} & \in \argmin_{x} \left\{ \delta_{C}\left(x\right) + \frac{\lambda}{2}\norm{Ax - u^{k}}^{2} \right\}, \label{AMNx} \\
			u^{k + 1} & \in \argmin_{u} \left\{ \delta_{Q}\left(u\right) + \frac{\lambda}{2} \norm{Ax^{k + 1} - u}^{2} \right\}. \label{AMNu}
		\end{align}}}
\medskip

	\begin{prop}
		Let $\left\{ \left(x^{k} , u^{k}\right) \right\}_{k \in \nn}$ be a bounded sequence generated by the AM Algorithm. Suppose that $C$ and $Q$ are semi-algebraic sets and $Q$ is convex. Then, the sequence $\Seq{u}{k}$ converges globally to some point $u^{\ast}$ such that for any limit point $x^{\ast}$ of the sequence $\Seq{x}{k}$, the pair $\left(x^{\ast} , u^{\ast}\right)$ is a critical point of $F_{1}$.
	\end{prop}
	\begin{proof}
		We first prove that the sequence $\left\{ \left(x^{k} , u^{k}\right) \right\}_{k \in \nn}$ is a a gradient-like descent sequence for minimizing $F_{1}$ according to Definition \ref{D:Gradlike}. In order to prove condition (C1) we use the fact that the function $u \rightarrow F\left(x^{k + 1} , u\right)$ is strongly convex with parameter $\lambda$. Thus, by the definition of strong convexity and the update step \eqref{AMNu}, we obtain that
		\begin{equation*}
			F\left(x^{k + 1} , u^{k + 1}\right) + \frac{\lambda}{2}\norm{u^{k + 1} - u^{k}}^{2} \leq F\left(x^{k + 1} , u^{k}\right).
		\end{equation*}		
		From the optimality of $x^{k + 1}$ in \eqref{AMNx} we obtain
		\begin{equation*}
			F\left(x^{k + 1} , u^{k}\right) \leq F\left(x^{k} , u^{k}\right).
		\end{equation*}
		Combining both inequalities yields the desired result with $\rho_{1} = \lambda/2$.
	
		For proving condition (C2) we use the optimality conditions associated with the two steps \eqref{AMNx} and \eqref{AMNu}, meaning
		\begin{equation*}
			\lambda A^{T}\left(u^{k} - Ax^{k + 1}\right) \in \partial \delta_{C}\left(x^{k + 1}\right),	
		\end{equation*}
		and
		\begin{equation*}
			\lambda\left(Ax^{k + 1} - u^{k + 1}\right) \in \partial \delta_{Q}\left(u^{k + 1}\right).	
		\end{equation*}
		Therefore, by subtracting $u^{k + 1}$ from both sides of the first inclusion, we arrive at
		\begin{equation*}
			\lambda A^{T}\left(u^{k} - u^{k + 1}\right) \in \partial \delta_{C}\left(x^{k + 1}\right) + \lambda A^{T}\left(Ax^{k + 1} - u^{k + 1}\right),
		\end{equation*}
		and
		\begin{equation*}
			\bo \in \partial \delta_{Q}\left(u^{k + 1}\right) + \lambda\left(u^{k + 1} - Ax^{k + 1}\right).
		\end{equation*}
		This proves that $w^{k + 1} := \left(\lambda A^{T}\left(u^{k} - u^{k + 1}\right) , \bo\right) \in \partial F\left(x^{k + 1} , u^{k + 1}\right)$. Hence, the assertion follows with $\rho_{2} = \lambda\norm{A}$.
	
		Lastly, condition (C3) easily follows since on the generated sequence we have that $F_{1}\left(x^{k} , u^{k}\right) = \left(\lambda/2\right)\norm{Ax^{k} - u^{k}}^{2}$. Therefore $F_{1}$ is continuous on the generated sequence, which proves the required assertion.
		
		Now, from Theorem \ref{T:AbstrGlob} we get the desired result.
	\end{proof}

\subsection{Convergence Analysis of the Lagrangian-Based Algorithms} \label{SSec:Lagrangian}
	As mentioned above, the proof methodology that was used to derive the convergence results of Sections \ref{SSec:CQ-SF1P} and \ref{SSec:AM-SF1P} can not be used to analyze the Lagrangian-based algorithms that were presented in Section \ref{Sec:Models}. However, for this class of algorithms we can follow other recent proof methodology of \cite{BST2018}, which was the first to propose a proof methodology for Lagrangian-based methods in the non-convex setting. To this end, they defined the concept of Lagrangian sequences, which is in some sense an extension of the notion of gradient-like descent sequence to the setting of primal-dual methods. Here we will show that the Proximal ADMM and the Linearized Proximal ADMM presented above indeed generate Lagrangian sequences and therefore converges to critical points of the optimization model (SF4), as was proved in \cite[Theorem 2, p. 1218]{BST2018}. It should be noted that the results of \cite{BST2018} are valid also in the case where the mapping $A$ is not necessarily linear.
\medskip

	Recalling that we focus on the formulation (SF4), where the set $C$ is convex and $Q$ non-convex, via its split form
	\begin{equation} \label{SF1-Split}
		\min_{x \in \real^{n}} \left\{ F_{2}\left(x\right) := d_{C}^{2}\left(x\right) + \delta_{Q}\left(Ax\right) \right\} = \min_{x \in \real^{n}, u \in \real^{m}} \left\{ d_{C}^{2}\left(x\right) + \delta_{Q}\left(u\right) : \, Ax = u \right\},
	\end{equation}
	and with the following augmented Lagrangian
	\begin{equation*}
		L_{\rho}\left(x , u , y\right) = d_{C}^{2}\left(x\right) + \delta_{Q}\left(u\right) + \act{y , Ax - u} + \frac{\rho}{2}\norm{Ax - u}^{2},
	\end{equation*}
	where $y \in \real^{m}$ is the multiplier which associated with the linear constraint $Ax = u$, and $\rho > 0$ is a penalization parameter.
\medskip

	We start by presenting a method which unifies both the Proximal ADMM and the Linearized Proximal ADMM presented above in Section \ref{SSec:ModelSF4}. We define the weighted norm of a vector $v$ with respect to a matrix $M \succ 0$ by $\norm{v}_{M}^{2} := \act{v , Mv}$.
\medskip

	\noindent\fbox{\parbox{\textwidth}{\textbf{Weighted Proximal ADMM for (SF4)} \\
		\textbf{Initialization:} Choose $x^{0} \in \real^{n}$ and $u^{0} , y^{0} \in \real^{m}$, fix a matrix $N \succ 0$. \\
		\textbf{General Step:} For $k = 0 , 1 , 2 , \ldots$ compute
		\begin{align}
			u^{k + 1} & \in \argmin_{u} \left\{ \delta_{Q}\left(u\right) + \act{y^{k} , Ax^{k} - u} + \frac{\rho}{2}\norm{Ax^{k} - u}^{2} \right\}, \label{WPADMMu} \\
			x^{k + 1} & \in \argmin_{x} \left\{ d_{C}^{2}\left(x\right) + \act{y^{k} , Ax  - u^{k + 1}} + \frac{\rho}{2} \norm{Ax - u^{k + 1}}^{2} + \frac{1}{2}\norm{x - x^{k}}_{N}^{2} \right\}, \label{WPADMMx} \\
			y^{k + 1} & = y^{k} + \rho\left(Ax^{k+1} - u^{k+1}\right). \label{WPADMMy}
		\end{align}}}
\medskip

	It is easy to verify that the Proximal ADMM is recovered with $N = \tau I$ for some $\tau > 0$, whereas the Linearized Proximal ADMM Algorithm is recovered when $N = \tau I - \rho A^{T}A$ for some $\tau > \rho\lambda_{max}\left(A^{T}A\right)$. Therefore, we will analyze only the Weighted Proximal ADMM and prove converges to critical points of the Problem (SF4). 
\medskip

	We therefore prove with the following result, which guarantees the global convergence of both algorithms.
	\begin{prop}
		Let $\left\{ \left(x^{k} , u^{k} , y^{k}\right)\right\}_{k \in \nn}$ be a bounded sequence generated by the Weighted Proximal ADMM Algorithm. Suppose that $C$ and $Q$ are semi-algebraic sets and $C$ is convex. Then, the sequence $\Seq{x}{k}$ converges globally to some point $x^{\ast}$ which is a critical point of $F_{2}$.
	\end{prop}
	\begin{proof}
		We first prove that the sequence $\left\{ \left(x^{k} , u^{k} , y^{k}\right) \right\}_{k \in \nn}$ is a Lagrangian sequence as defined in \cite{BST2018}.This requires to prove four conditions. In order to prove the first condition we use the optimality of $u^{k + 1}$ and $x^{k + 1}$ in \eqref{WPADMMu} and \eqref{WPADMMx}, respectively, to obtain that
		\begin{equation*}
			L_{\rho}\left(x^{k} , u^{k + 1} , y^{k}\right) \leq L_{\rho}\left(x^{k} , u^{k} , y^{k}\right),
		\end{equation*}
		and
		\begin{equation*}
			L_{\rho}\left(x^{k + 1} , u^{k + 1} , y^{k}\right) + \frac{1}{2}\norm{x^{k + 1} - x^{k}}_{N}^{2} \leq L_{\rho}\left(x^{k} , u^{k + 1} , y^{k}\right).
		\end{equation*}
		Combining both inequalities yields that
		\begin{equation*}
			\frac{\lambda_{min}\left(N\right)}{2}\norm{x^{k + 1} - x^{k}}^{2} \leq L_{\rho}\left(x^{k} , u^{k} , y^{k}\right) - L_{\rho}\left(x^{k + 1} , u^{k + 1} , y^{k}\right), \quad \forall \,\, k \in \nn,
		\end{equation*}		
		which proves the first condition.
	
		For proving the second and third conditions we use the optimality conditions associated with the two steps \eqref{WPADMMu} and \eqref{WPADMMx}, meaning
		\begin{equation*}
			y^{k} + \rho\left(Ax^{k} - u^{k + 1}\right) \in \partial \delta_{Q}\left(u^{k + 1}\right),
		\end{equation*}
		and
		\begin{equation*}
			-A^{T}\left(y^{k} + \rho\left(Ax^{k + 1} - u^{k + 1}\right)\right) + N\left(x^{k} - x^{k + 1}\right) \in \nabla d_{C}^{2}\left(x^{k + 1}\right).
		\end{equation*}
		Therefore, by subtracting $\lambda Ax^{k + 1}$ from both sides of the first inclusion, we arrive at
		\begin{equation*}
			\rho A\left(x^{k} - x^{k + 1}\right) \in \partial \delta_{Q}\left(u^{k + 1}\right) - y^{k} - \rho\left(Ax^{k + 1} - u^{k + 1}\right) ,
		\end{equation*}
		and
		\begin{equation*}
			N\left(x^{k} - x^{k + 1}\right) \in \nabla d_{C}^{2}\left(x^{k + 1}\right) + A^{T}\left(y^{k} + \rho\left(Ax^{k + 1} - u^{k + 1}\right)\right).
		\end{equation*}
		Thus $\norm{\nabla_{x} L_{\rho}\left(x^{k + 1} , u^{k + 1} , y^{k}\right)} \leq \lambda_{max}\left(N\right)\norm{x^{k + 1} - x^{k}}$, which proves the second condition with $b = \lambda_{max}\left(N\right)$ and $v^{k + 1} :=  \rho A\left(x^{k} - x^{k + 1}\right) \in \partial_{u} L_{\rho}\left(x^{k + 1} , u^{k + 1} , y^{k}\right)$. Hence, the third condition follows.
	
		The forth condition easily follows since the set $Q$ is closed and therefore $\overline{u} \in Q$, which yields the desired result.
		
		Now, the result follows directly from \cite[Theorem 2, p. 1218]{BST2018}.
	\end{proof}

\section{Multiple-Sets Split Feasibility Problems} \label{Sec:Extensions}
	In this section we study Multiple-Sets Split Feasibility (MSSF) Problems and discuss how the above mentioned methods can be adjusted to solve this extension. Let $r$ be a natural number. Let $C$ and $Q_{j}$, $1 \leq j \leq r$, be non-empty and closed subsets of $\real^{n}$ and $\real^{m_{j}}$, respectively. Given also linear mappings $A_{j} : \real^{n} \rightarrow \real^{m}$, then the MSSF Problem seeks to find a point $x^{\ast}$ for which
	\begin{equation} \label{MSSFP}
		x^{\ast} \in C \text{ such that } A_{j}x^{\ast} \in Q_{j}, \,\, \forall \, 1 \leq j \leq r,
	\end{equation}
	we refer the reader for some papers that discuss this extension in the convex setting \cite{ms07,glt18,ree17,gmv19} (which study a more general version of the problem which also allows for multiple sets $C$). Note that, when $A_{j} = A$, $1 \leq j \leq r$, for some linear mapping $A$, then Problem \eqref{MSSFP} reduced to the following version:
	\begin{equation*}
		x^{\ast} \in C \text{ such that } Ax^{\ast} \in \cap _{j = 1}^{r} Q_{j}.
	\end{equation*}
	This problem can be also formulated using similar optimization models (like (SF1)--(SF4)), where each set can be translated using the indicator function or the squared distance function. Therefore, this can yield, in general, many optimization models, but as we discussed above depends on the properties of the set (like convexity) one should decide if to use the indicator function or the distance function. Let's take an example.
\medskip

	As mentioned above one of the optimization models that is very popular in this area is (SF3), and in case of a multiple sets it translates to.
	\begin{equation} \label{SF3MS}
		\min_{x \in \real^{n}} \left\{ \sum_{j = 1}^{r} d_{Q_{j}}^{2}\left(A_{j}x\right) + \delta_{C}\left(x\right) \right\} = \min_{x \in C} \left\{ \sum_{j = 1}^{r}d_{Q_j}^{2}\left(A_{j}x\right) \right\}.
	\end{equation}
	This model, makes sense when the set $C$ is not necessarily convex, while the sets $Q_{j}$, $1 \leq j \leq r$, are convex. The Projected Gradient can be applied for solving this problem, which yields the following extension of the CQ Algorithm
	\begin{equation*}
	x^{k + 1} = P_{C}\left(x^{k} - \frac{1}{\tau}\sum_{j = 1}^{r}A_{j}^{T}\left(A_{j}x^{k} - P_{Q_{j}}\left(A_{j}x^{k}\right)\right)\right),
	\end{equation*}
	for some step-size $\tau > 0$. This algorithm, as can be seen, tackles the sets $Q_{j}$, $1 \leq j \leq r$, in a simultaneous manner, which could result in a slow convergence. A cyclic based algorithm for solving this problem in the convex setting is presented in \cite{wx11}.
	\begin{rema}
		\begin{enumerate}[(i)]
			\item It should be noted that all the algorithms presented in Section \ref{Sec:Models} can be easily generalized to this version of MSSF problems in a simultaneous way similar to the above. The case where there exists multiple sets $C$ is much more complicated and deserves a special study especially in the non-convex setting.
			\item Since our analysis is based on non-linear optimization techniques, some of the results in this paper could also be applied for solving non-linear split feasibility problems. The proposed algorithms are different than the ones presented in Non-linear convex feasibility problems were studied for example in   \cite{gks14,xcyl18} and are less restrictive.
		\end{enumerate}
	\end{rema}
\bigskip

To conclude, in this paper, we have studied the split feasibility problem in the non-convex setting through the presentation of four different models that formulate the feasibility problem as an optimization problem using the indicator and/or squared distance functions. We discussed seven algorithms for tackling these formulations, where each one of them tackles the formulation which fits best in sense of the information on the involved sets $C$ and $Q$ (in this abstract paper we have only exploited the feature of convexity). This paper has two main goals: in the theoretical front, we have provided an up-to-date list of algorithms with the current known theoretical results as summarized in the table on page 2. From this table, one can find a comparison between the discussed algorithms in the sense of global convergence, all algorithms but Algorithm 1 have this property (the global convergence of Algorithms 1 seems to be very a challenging task for future research!). In the applied front, this paper provides a source for users that are interested in solving split feasibility problems in the non-convex world. By exploiting the special structure and data information of the particular sets $C$ and $Q$ (like convexity but others could be useful) one can find the algorithm that best fits the applications at hand.﻿

\bibliographystyle{plain+eid}
\nocite{*}
\bibliography{samplebib}

\end{document}